\documentclass{article}
\usepackage{amsmath}
\usepackage{amsfonts}
\usepackage{amssymb}
\usepackage{amsthm}
\usepackage{color}
\usepackage[color=green]{todonotes}
\usepackage{enumerate}
\usepackage{geometry}
\usepackage{listings}
\usepackage{microtype}
\usepackage{rotating}
\usepackage{textcomp}
\usepackage{makeidx}
\usepackage{framed}
\usepackage{array}
\usepackage[hidelinks]{hyperref}
\usepackage[nameinlink,capitalise]{cleveref}
\usepackage{authblk}
\usepackage{comment}
\usepackage[small]{titlesec}

\newcommand{\Z}{\mathbb{Z}}
\newcommand{\N}{\mathbb{N}}
\renewcommand{\P}{\mathbb{P}}
\newcommand{\E}{\mathbb{E}}
\newcommand{\cR}{\mathcal{R}}
\newcommand{\ttau}{\widetilde{\tau}}
\newcommand{\tsigma}{\widetilde{\sigma}}
\newcommand{\cF}{\mathcal{F}}

\DeclareMathOperator{\cov}{cov}
\DeclareMathOperator{\hit}{hit}
\DeclareMathOperator{\mix}{mix}
\DeclareMathOperator{\SRW}{SRW}

\DeclareMathOperator{\e}{e}
\DeclareMathOperator{\TV}{TV}

\theoremstyle{plain} 
\newtheorem{theorem}{Theorem}[section]
\newtheorem*{theorem*}{Theorem}

\newtheorem*{proposition*}{Proposition}
\newtheorem{lemma}[theorem]{Lemma}
\newtheorem*{lemma*}{Lemma}

\theoremstyle{definition} 
\newtheorem{corollary}[theorem]{Corollary}
\newtheorem*{corollary*}{Corollary}
\newtheorem{remark}[theorem]{Remark}
\newtheorem*{remark*}{Remark}

\newtheorem*{fact*}{Fact}

\newtheorem*{definition*}{Definition}

\newtheorem*{example*}{Example}

\newtheorem*{idea*}{Idea}

\newtheorem*{property*}{Property}

\numberwithin{equation}{section}

\usepackage[english]{babel}

\begin{document}

\title{Cover times for random walk on dynamical percolation}
\date{\vspace{-5ex}}

\author[1]{Maarten Markering}
\affil[1]{University of Cambridge} 
\maketitle  
\let\thefootnote\relax\footnotetext{Date: \today}
\let\thefootnote\relax\footnotetext{Email: mjrm2@cam.ac.uk}
\let\thefootnote\relax
\let\thefootnote\relax

\begin{abstract}
We study the cover time of random walk on dynamical percolation on the torus $\Z_n^d$ in the subcritical regime. In this model, introduced by Peres, Stauffer and Steif in \cite{Peres2015Random}, each edge updates at rate $\mu$ to open with probability $p$ and closed with probability $1-p$. The random walk jumps along each open edge with rate $1/(2d)$. We prove matching (up to constants) lower and upper bounds for the cover time, which is the first time that the random walk has visited all vertices at least once. Along the way, we also obtain a lower bound on the hitting time of an arbitrary vertex starting from stationarity, improving on the maximum hitting time bounds from \cite{Peres2015Random}.
\end{abstract}

\section{Introduction}

In this paper, we study \emph{random walk on dynamical percolation} on $\Z_n^d$, first introduced by Peres, Stauffer and Steif \cite{Peres2015Random}. We denote the vertex set $\{0,\ldots,n-1\}^d$ of $\Z_n^d$ by $\Z_n^d$ and the edge set by $E(\Z_n^d)$. The model is defined as follows. Each edge refreshes at rate $\mu=\mu_n$. It becomes open with probability $p$ and closed with probability $1-p$. This is the \emph{dynamical percolation model}. We denote the edge set at time $t$ by $\eta_t\in\{0,1\}^{E(\Z_n^d)}$ and call it the \emph{environment}. Here $\eta_t(e)=0$ corresponds to the edge $e$ being closed and $\eta_t(e)=1$ means it is open. We define a continuous time random walk $X=(X_t)_{t\geq0}$ that moves as follows. It chooses one of the $2d$ edges incident to the walker uniformly at rate 1. If the edge is open, it jumps across this edge. If the edge is closed, the walk stays in place. The \emph{full system} is denoted by
\begin{equation}
    (M_t)_{t\geq0}:=((X_t,\eta_t))_{t\geq0}.
\end{equation}
The process $M$ is a reversible Markov process with stationary distribution $u\times\pi_p$, where $u$ is the uniform measure on $\Z_n^d$, and $\pi_p$ is the product Bernoulli measure on $E(\Z_n^d)$. It is important to note that the walk process $(X_t)_{t\geq0}$ is itself not Markovian.

Peres, Stauffer and Steif started the investigation of the model in \cite{Peres2015Random}. Their analysis focused mainly on the regime $p\in(0,p_c(d))$, where $p_c(d)$ is the critical value for bond percolation on $\Z^d$. Their main result was a bound on the mixing time
\begin{equation}
    t_{\mix}(\varepsilon)=\inf\{t\geq0:\max_{x,\eta_0}\|\P_{x,\eta_0}((X_t,\eta_t)=(\cdot,\cdot))-u\times\pi_p\|_{\TV}\leq\varepsilon\}.
\end{equation}

\begin{theorem}[{\cite[Theorem 1.2]{Peres2015Random}}]\label{theorem:mixingupperbound}
    For all $d\geq1$ and $p\in(0,p_c(\Z_n^d))$, there exists a positive constant $C=C(d,p)$ such that for all $n\in\N$ and $\mu\leq1$, we have
\begin{equation}
    t_{\mix}\leq\frac{Cn^2}{\mu}.
\end{equation}
\end{theorem}

They also obtained upper bounds on the maximum expected hitting time of the walk started from the stationary environment. For $y\in\Z_n^d$, let $\sigma_y$ be the first time that the walk visits $y$:
\begin{equation}
    \sigma_y:=\inf\{t\geq0:X_t=y\}
\end{equation}
and define
\begin{equation}
    t_{\hit}:=\max_{x,y\in\Z_n^d}\E_{x,\pi_p}[\sigma_{y}].
\end{equation}

\begin{theorem}[{\cite[Theorem 1.12]{Peres2015Random}}]\label{th:hittingupper}
For all $p\in(0,p_c(d))$, there exist positive constants $C_1=C_1(d,p)$ and $C_2=C_2(d,p)$ such that for all $n\in\N$ and $\mu\leq1$:
\begin{equation}
    \begin{aligned}
    C_1\frac{n^2}{\mu}\leq&\,t_{\hit}\leq C_2\frac{n^2}{\mu}, && d=1,\\
    C_1\frac{n^2\log n}{\mu}\leq&\,t_{\hit}\leq C_2\frac{n^2\log n}{\mu}, && d=2,\\
    C_1\frac{n^d}{\mu}\leq&\,t_{\hit}\leq C_2\frac{n^d}{\mu}, && d\geq3.
    \end{aligned}
\end{equation}
\end{theorem}
Although sharp asymptotic bounds for the mixing and maximum hitting times have been shown in the subcritical regime, the supercritical regime is not as well understood. A lower bound for the mixing time of order $n^2+\frac{1}{\mu}$ was shown in \cite{Peres2015Random} and it was conjectured that this should also be the correct order for the upper bound. The matching upper bound was shown up to polylogarithmic factors in part of the supercritical regime by Peres, Sousi and Steif in \cite{Peres2018Quenched, Peres2017Mixing}.

Hermon and Sousi in \cite{Hermon2020Comparison} studied the model for general graphs and all parameter values $p$. They obtained mixing and hitting time upper bounds in terms of the corresponding quantities for the simple random walk on the static graph. 

Currently, there are no known results on the \emph{cover time}, which is the first time that the walk has visited all states at least once. This is the focus of the present paper. To define the cover time, for $s,t\geq0$, $s\leq t$, let $\cR[s,t]$ be the set of vertices that the walk $X$ has visited in the time interval~$[s,t]$, i.e.,
\begin{equation}
    \cR[s,t]:=\{y\in\Z_n^d:\exists r\in[s,t]\text{ such that } X_r=y\}.
\end{equation}
Let $\tau_{\cov}$ be the first time that all vertices in $\Z_n^d$ have been visited by $X$, so
\begin{equation}
    \tau_{\cov}:=\inf\{t\geq0:\cR[0,t]=\Z_n^d\}.
\end{equation}
We define the \emph{maximum expected cover time} $t_{\cov}$ as the expectation of $\tau_{\cov}$ starting from the worst possible vertex and environment:
\begin{equation}
    t_{\cov}:=\max_{x\in\Z_n^d,\eta_0\in\{0,1\}^{E(\Z_n^d)}}\E_{x,\eta_0}[\tau_{\cov}].
\end{equation}
Here $\E_{x,\eta_0}$ denotes the expectation of the full system $(M_t)_{t\geq0}$ when $M_0=(x,\eta_0)$. We define $\P_{x,\eta_0}$ analogously. More generally, if $\rho$ and $\pi$ are distributions on $\Z_n^d$ and $E(\Z_n^d)$ respectively, we write $\P_{\rho,\pi}$ and $\E_{\rho,\pi}$ for the law and expectation of $(M_t)_{t\geq0}$ when $M_0\sim(\rho,\pi)$.

Recall $p_c(d)$ is the critical probability for bond percolation on $\Z^d$. We are now ready to state the main theorem of this paper.
\begin{theorem}[Cover time bounds]\label{th:maincover}
For $d\geq 1$ and $p\in (0,p_c(d))$, there exist constants $C_1=C_1(d,p)$ and $C_2=C_2(d,p)$ such that for all $n\in\N$ and $\mu\leq1$ random walk on dynamical percolation on $\Z_n^d$ with parameters $\mu$ and $p$ satisfies
\begin{equation}
    \begin{aligned}
        \frac{C_1 n^2}{\mu}\leq& t_{\cov}\leq\frac{C_2n^2}{\mu}, && d=1,\\
        \frac{C_1 n^2(\log n)^2}{\mu}\leq& t_{\cov}\leq\frac{C_2n^2(\log n)^2}{\mu}, && d=2,\\
        \frac{C_1 n^d\log n}{\mu}\leq& t_{\cov}\leq\frac{C_2n^d\log n}{\mu}, && d\geq3.
    \end{aligned}
\end{equation}
\end{theorem}

In order to prove the lower bound on the cover time in dimensions $d\geq3$ in Theorem \ref{th:maincover}, a key step is to obtain a lower bound on the expected hitting time of any vertex when the initial environment is very close to a stationary one as we now define. For $x\in\Z_n^d$, define the probability measure $\pi^x_p$ on~$\{0,1\}^{\Z_n^d}$ to be the measure $\pi_p$ conditioned on $\eta(e)=0$ for all $e$ incident to $x$, i.e.\ $\pi^x_p(\eta(e)=0)=1$ if $e$ is incident to $x$ and $\pi^x_p(\eta(e)=0)=1-p=1-\pi_p^x(\eta(e)=1)$ otherwise. 
The measure $\pi_p^x$ was also used in~\cite{Peres2015Random} and it will play an important role throughout the entire proof of the cover time lower bound in dimensions $d\geq 3$.

For $y\in\Z_n^d$, let $\sigma_y$ be the first time that the walk visits $y$:
\begin{equation}
    \sigma_y:=\inf\{t\geq0:X_t=y\}.
\end{equation}
In the following theorem we obtain a lower bound on the expected hitting time of any vertex which is of the same order as $t_{\hit}$ (from Theorem~\ref{th:mainhitting}) when the initial environment is $\pi_p^x$ for some $x$.

\begin{theorem}[Hitting time lower bound]\label{th:mainhitting}
For all $d\geq 3$ and $p\in(0,p_c(d))$, there exists a constant $C=C(d,p)>0$ such that for all $n\in\N$, $\mu\leq1$ and $x,y\in\Z_n^d$ with $x\neq y$,
\begin{equation}
    \E_{x,\pi^x_{p}}[\sigma_y]\geq C\frac{n^d}{\mu}.
\end{equation}
\end{theorem}

An immediate corollary of the above is a lower bound on the expected hitting time of any vertex when the initial environment is stationary. 
\begin{corollary}\label{cor:mainhitting}
For all $d\geq 3$ and $p\in(0,p_c(d))$, there exists a constant $C=C(d,p)>0$ such that for all $n\in\N$, $\mu\leq1$ and $x,y\in\Z_n^d$ with $x\neq y$,
\begin{equation}
    \E_{x,\pi_p}[\sigma_y]\geq C\frac{n^d}{\mu}.
\end{equation}
\end{corollary}

\subsection{Overview of the proof and outline}
A key tool in the proof for dimensions $d\geq3$ are lower and upper bounds for the cover time in terms of hitting times due to Matthews \cite{Matthews1988Covering}, which we recall here. For a set of vertices $A\subset\Z_n^d$, define
\begin{equation}
    t^{\SRW}_{\hit}:=\max_{\substack{x,y\in \Z_n^d}}\E^{\SRW}_x[\sigma_y],\qquad t_{A}^{\SRW}:=\min_{\substack{x,y\in A\\x\neq y}}\E^{\SRW}_x[\sigma_y].
\end{equation}
Then Matthews' bound states
\begin{equation}\label{eq:matthewsSRW}
    t_A^{\SRW}\left(1+\ldots+\frac{1}{|A|-1}\right)\leq t_{\cov}^{\SRW}\leq t_{\hit}^{\SRW}\left(1+\ldots+\frac{1}{n^d}\right).
\end{equation}
Here $\E_x^{\SRW}$ and $t_{\cov}^{\SRW}$ are the expectation and cover time for the simple random walk. It is straightforward to adapt the upper bound to the setting of dynamical percolation, since we can derive a hitting time upper bound for all vertices and all starting environments from Theorem \ref{th:hittingupper} for $d\geq2$. This is shown in Section \ref{sec:upperbound}.

The lower bounds are much more challenging. In Section \ref{sec:preliminaries}, we recall the sequence of \emph{regeneration times} $(\ttau_k)_{k\in\N}$ defined in \cite[Section 6]{Peres2015Random}, which will play a crucial role throughout the proof. The regeneration times are constructed such that at time $\ttau_k$, the environment has distribution $\pi_p^x$ conditional on $X_{\ttau_k}=x$, and $(X_{\ttau_k})_{k\in\N}$ is a symmetric random walk. Furthermore, the number of vertices visited in between regeneration times has exponential tails.

We then prove Theorem \ref{th:mainhitting} in Section \ref{sec:hittingtime}. Using that the walk $(X_{\ttau_k})_{k\in\N}$ is \emph{locally transient}, we show that with probability bounded away from 0, the hitting time of any vertex $y\in\Z_n^d$ is larger than the mixing time of $X$. After mixing, it takes time of order $n^d/\mu$ to hit $y$, since $X$ only visits a bounded number of vertices in $\frac{1}{\mu}$ steps. Since $(X_{\ttau_k})_{k\in\N}$ is a symmetric random walk, we can use standard results to control its behaviour. One of the main technical difficulties is controlling what happens in between regeneration times. In particular, we want to lower bound the probability that~$X$ does not hit $y$ in between regeneration times. 

In Section \ref{sec:matthews}, we transfer the hitting time lower bound from Theorem \ref{th:mainhitting} to a cover time lower bound using an adaptation of Matthews' method. In contrast to the upper bound, the adaptation is not straightforward, because bounding by the minimum hitting time over all environments will not yield a sharp bound. Instead, we want to lower bound the cover time in terms of the hitting time started from the environment $\pi_p^x$. We overcome this problem by choosing the set $A$ from \eqref{eq:matthewsSRW} to be the set of vertices whose coordinates are multiples of $\lfloor\sqrt{n}\rfloor$, so the vertices in $A$ are a distance of at least $\sqrt{n}$ apart. Since the maximum number of vertices visited in between times $\ttau_j$ and $\ttau_{j+1}$ is at most of order $\log n$, there will be a time $\ttau_j$ in between consecutive visits to distinct vertices of~$A$ with very high probability. At that time, the environment has distribution $\pi_{p}^x$, and we can apply the hitting time lower bound of Theorem \ref{th:mainhitting}.

For $d=2$, the random walk is not locally transient, so there is no hitting time lower bound that is uniform over all vertices like in Theorem \ref{th:mainhitting}. Instead, we couple the walk along regeneration times to Brownian motion using a refined \emph{strong approximation} theorem. Again, we have to control what happens in between regeneration times. We prove this in Section \ref{sec:d2}.

Finally, in Section 6, we treat the one-dimensional case.

\section{Upper bound}\label{sec:upperbound}
In this section, we prove the upper bounds of Theorem \ref{th:maincover} for dimensions $d\geq2$ using an adaptation of Matthews' upper bound on cover times for Markov chains. The adaptation of the proof to the setting of dynamical percolation is straightforward, but we include it here for the reader's convenience. 

\begin{lemma}[Matthews' upper bound]\label{lemma:matthewsupper}
For all $d\geq 1$, $p\in(0,1)$, $n\in\N$ and $\mu>0$, random walk on dynamical percolation on $\Z_n^d$ with parameters $\mu$ and $p$ satisfies
\begin{equation}
    t_{\cov}\leq t_{\hit}\left(1+\ldots+\frac{1}{n^d}\right).
\end{equation}
\end{lemma}
\begin{proof}
This proof follows the exposition of the proof of Matthews' bound from~\cite[Theorem 11.2]{levin2017markov}.
We label the vertices of $\Z_n^d$ as $\{1,\ldots,n^d\}$, and let $\phi$ be a uniformly random permutation of $\{1,\ldots,n^d\}$ that is independent of $((X_t,\eta_t))_{t\geq0}$. Set $T_0:=0$ and define $T_k$ to be the first time that the random walk $X$ visits the set of vertices $\{\phi(1),\ldots,\phi(k)\}$. We let $L_k=X_{T_k}$. Note that $L_k\neq\phi(k)$ if and only if the vertex $\phi(k)$ was already visited by $X$ by time $T_{k-1}$, in which case $T_{k}=T_{k-1}$. By symmetry and independence of $\phi$ from $X$, we have $\P(L_k=\phi(k))=\frac{1}{k}$ and
\begin{equation}
    \E_{x,\eta_0}[T_k-T_{k-1}|L_k=\phi(k)]\leq\max_{x',\eta'_0}\E_{x',\eta_0'}[\E_{x',\eta'_0}[\sigma_{\phi(k)}]|\phi]\leq t_{\hit}.
\end{equation}
Hence,
\begin{equation}
    \begin{split}
        t_{\cov}\leq&\sum_{k=1}^{n^d}\max_{x,\eta_0}\left\{\E_{x,\eta_0}[T_k-T_{k-1}|L_k=\sigma(k)]\P(L_k=\phi(k))\right.\\
        &\left.+\E_{x,\eta_0}[T_k-T_{k-1}|L_k\neq\phi(k)]\P(L_k\neq\phi(k))\right\}\leq t_{\hit}\sum_{k=1}^{n^d}\frac{1}{k}
    \end{split}
\end{equation}
and this concludes the proof.
\end{proof}

\begin{proof}[Proof of upper bounds in Theorem \ref{th:maincover} for $d\geq2$]
    This follows immediately from Theorem \ref{th:hittingupper} and Lemma \ref{lemma:matthewsupper}. Note that Theorem \ref{th:hittingupper} was only stated for stationary starting environments, but the result carries over to arbitrary starting environments since the time until every edge updates is with high probability of order $\log n^d/\mu$ which is much smaller than $ t_{\hit}$.
\end{proof}

\section{Regeneration times}\label{sec:preliminaries}
In this section we recall the definition of regeneration times introduced by Peres, Stauffer and Steif in~\cite[Section 6]{Peres2015Random}. Observing the walk along the regeneration times will play an important role in the proofs. Our goal is to define a sequence of \emph{regeneration times} $(\ttau_{k})_{k\in\N_0}$ such that, if we condition on $X_{\ttau_k}=x$ for some vertex $x\in\Z_n^d$, then $\eta_{\ttau_k}\sim\pi_{p}^x$ and $(X_{\ttau_k})_{k\in\N_0}$ is a symmetric Markovian random walk. Furthermore, $(\ttau_{k+1}-\ttau_k)_k$ will be an i.i.d.\ sequence with expectation of order $\frac{1}{\mu}$. Using these properties, we are able to transfer classical results on random walks on $\Z_n^d$ to bounds on hitting and cover times of the random walk on dynamical percolation at the cost of a factor $\frac{1}{\mu}$.

We first define the random process $(A_t)_{t\geq0}$ on $\E(\Z_n^d)$. The set $A_t$ should be thought of as the set of edges that the random walk $X$ has information on at time $t$. Let $A_0$ be the set of edges incident to $X_0$. If at time $t$, $X$ jumps to a vertex $v$, then all edges incident to $v$ are added to $A_{t^-}$. If at time~$t$, an edge $e\in A_{t^-}$ refreshes its state and $e$ is not incident to $X_t$, then $e$ is removed from~$A_{t^-}$. Let~$(\cF^{\star}_t)_{t\geq0}$ be the natural filtration associated with $(X_t)_{t\geq0}$, $(A_t)_{t\geq0}$ and $((\eta_t)|_{A_t})_{t\geq0}$, where~$(\eta_t)|_{A_t}$ is the restriction of $\eta_t$ to $A_t$. So this filtration only contains knowledge of the walk's position, whether an edge is in~$A_t$ and the state of the edges in $A_t$. The following result from~\cite[Proposition 6.11]{Peres2015Random} states that the sets $A_t$ tend to decrease in size on a time scale of order $\frac{1}{\mu}$.

\begin{lemma}\label{lemma:decayAfromPSS}
There exist constants $C=C(d,p)>0$ and $C_A=C_A(d,p)>0$ such that for all $n\in\N$, $\mu\leq1$, $x\in\Z_n^d$ and $t\geq0$,
\begin{equation}
    \E_{x,\pi_{p}^x}\left[|A_{t+\frac{C_A}{\mu}}|\mid\cF_t^{\star}\right]\leq\frac{|A_t|}{4}+C\log|A_t|.
\end{equation}
\end{lemma}

We first define an increasing sequence of stopping times $(\tau_k)_{k\in\N_0}$ as follows. Let $\tau_0=0$ and for $k\geq1$, define
\begin{equation}
    \tau_k=\inf\{j>\tau_{k-1}:|A_{\frac{jC_A}{\mu}}|=2d\text{ and }\eta_{\frac{jC_A}{\mu}}(e)=0\text{ for all } e\in A_{\frac{jC_A}{\mu}}\},
\end{equation}
where $C_A$ is as in Lemma \ref{lemma:decayAfromPSS}. Define
\begin{equation}\label{eq:deftautilde}
    \ttau_k:=\frac{C_A}{\mu}\tau_k.
\end{equation}
So at each time $\ttau_k$, $A_{\ttau_k}$ only consists of the edges incident to $X_{\ttau_k}$, which at that time are all closed. The rest of the edges have Bernoulli distribution. Thus, conditioning on $X_{\ttau_k}=x$ for some $x\in\Z_n^d$, we have $\eta_{\ttau_k}\sim\pi^x_{p}$. Furthermore, starting the process according to~$\delta_x\times\pi^x_{p}$, the increments $(\ttau_{k}-\ttau_{k-1})$ are i.i.d.\ and the process $(X_{\ttau_k})_{k\in\N_0}$ is clearly a Markovian symmetric random walk on $\Z_n^d$ with uncorrelated coordinates. Also,~$X$ satisfies the strong Markov property at the stopping times $\ttau_k$.

We now state two results \cite[Theorem 6.18 and Lemma 7.4]{Peres2015Random} that will be useful in the proofs. The first says that $\tau_1$ is uniformly bounded, which implies that for each $k$ the difference $\ttau_{k}-\ttau_{k-1}$ is typically of order $\frac{1}{\mu}$. The second says that the number of vertices visited between times $\ttau_{k}$ and~$\ttau_{k+1}$ has exponential tails.
\begin{lemma}\label{lemma:taubounded}
For all $d\geq1$ and $p\in(0,p_c(d))$, there exists a constant $C=C(d,p)>0$ such that for all $n\in\N$, $\mu\leq1$ and $x\in\Z_n^d$,
\begin{equation}
    \E_{x,\pi_{p}^x}[\tau_1]\leq C.
\end{equation}
\end{lemma}

\begin{lemma}\label{lemma:exponentialtailsR}
For all $d\geq1$ and $p\in(0,p_c(d))$, there exist constants $C_{\cR,1}=C_{\cR,1}(d,p)$, $C_{\cR,2}=C_{\cR,2}(d,p)>0$ and $C=C(d,p)>0$ such that for all $n\in\N$, $\mu\leq1$, $x\in\Z_n^d$ and $k\geq0$,
\begin{equation}
    \E_{x,\pi_{p}^x}\left[e^{C_{\cR,1}|\cR[\ttau_k,\ttau_{k+1}]|}\right]\leq C_{\cR,2}
\end{equation}
and in particular,
\begin{equation}
    \E_{x,\pi_{p}^x}\left[|\cR[\ttau_k,\ttau_{k+1}]|\right]\leq C.
\end{equation}
\end{lemma}
Finally we recall the following result from~\cite[Lemma 7.2]{Peres2015Random}.

\begin{lemma}[Local Central Limit Theorem]\label{lemma:localCLT}
For all $d\geq1$ and $p\in(0,p_c(d))$, there exists a constant $C=C(d,p)>0$, such that for all $n\in\N_0$, $\mu\leq1$, $x,y\in\Z_n^d$ and $k\geq0$,
\begin{equation}
    \P_{x,\pi_{p}^x}(X_{\ttau_k}=y)\leq C\left(\frac{1}{k^{d/2}}\vee\frac{1}{n^d}\right)
\end{equation}
and for all $k\geq\frac{n^2}{2}$,
\begin{equation}
    \P_{x,\pi_{p}^x}(X_{\ttau_k}=y)\geq\frac{1}{C n^d}.
\end{equation}
\end{lemma}

\section{Cover time lower bound for $d\geq3$}\label{sec:lowerbound}

In this section we prove the lower bound of Theorem~\ref{th:maincover} for $d\geq3$. An important ingredient in the proof is Theorem~\ref{th:mainhitting} which we prove in Section~\ref{sec:hittingtime}. Then using Matthews' method, we are able to deduce the lower bound of Theorem \ref{th:maincover} for $d\geq3$ in Section \ref{sec:matthews}.

\subsection{Hitting times}\label{sec:hittingtime}
We prove Theorem \ref{th:mainhitting} in several steps. First, in Lemma~\ref{lemma:hittingboundary} we show that with probability bounded away from~0, the random walk reaches distance $\frac{n}{2}$ from $y$ before hitting $y$ itself. In Lemma \ref{lemma:hityuntilmixing}, we then show that this implies that with probability bounded away from 0, the random walk does not visit $y$ up until $2t_{\mix}$. Finally, in Lemma \ref{lemma:hityaftermixing}, we consider times after $2t_{\mix}$.

\begin{lemma}\label{lemma:hittingboundary}
For all $d\geq 3$ and $p\in(0,p_c(d))$, there exists a positive constant $c=c(d,p)$ such that for all $n\in\N$, $\mu\leq1$ and $x,y\in\Z_n^d$ such that $x\neq y$, we have
\begin{equation}
    \P_{x,\pi^x_{p}}(\widetilde{\sigma}<\sigma_y)\geq c,
\end{equation}
where $\widetilde{\sigma}:=\inf\{\ttau_k\geq0\mid d(X_{\ttau_k},y)\geq\frac{1}{2} n\}$ and $d(\cdot,\cdot)$ is the graph distance metric.
\end{lemma}

The proof of this lemma is done in a few steps. The first step consists of bounding the probability that the walk visits $y$ during an interval $[\ttau_k,\ttau_{k+1}]$. This probability is small if $X_{\ttau_k}$ is far away from $y$, because $|\cR[\ttau_k,\ttau_{k+1}]|$ is small. Furthermore, for large enough $k$ (say $k\geq K=K(d,p)$), the probability that $(X_{\ttau_k})$ is close to $y$ is small by the local Central Limit Theorem. So with high probability, $X$ will escape to distance $\frac{1}{2}n$ before hitting $y$ after $\ttau_{K}$. We also need to control the probability of $X$ hitting $y$ up until time $\ttau_K$. We show that this probability is very small if $x$ is at least some distance $M=M(d,p)$ from $y$. If $x$ is at distance less than $M$ from $y$, we show that with probability bounded away from~0, the random walk starting from~$x$ reaches distance $M$ before hitting $y$ and it then escapes to distance $\frac{1}{2}n$.

\begin{proof}
For $k\geq 0$, we let $E_k$ be the event that $X$ visits $y$ during $[\tilde{\tau}_k,\tilde{\tau}_{k+1}]$. Then
\begin{equation}
    \P_{x,\pi^x_{p}}(E_k)\leq\P_{x,\pi^x_{p}}(E_k|d(X_{\ttau_k},y)\geq k^{\frac{1}{4d}})+\P_{x,\pi^x_{p}}(d(X_{\ttau_k},y)\leq k^{\frac{1}{4d}}).
\end{equation}
Firstly, by Lemma~\ref{lemma:exponentialtailsR} we have 
\begin{equation}\label{eq:firstterm}
    \begin{split}
        \P_{x,\pi^x_{p}}(E_k|d(X_{\ttau_k},y)\geq k^{\frac{1}{4d}})\leq&\P_{x,\pi^x_{p}}(|\cR[\ttau_k,\ttau_{k+1}]|\geq k^{\frac{1}{4d}})\leq C_{\cR,2}\e^{-C_{\cR,1}k^{\frac{1}{4d}}}.
    \end{split}
\end{equation}
By Lemma \ref{lemma:localCLT}, it follows that there exists a constant $C_1=C_1(d,p)>0$ such that for $k\leq n^2$,
\begin{equation}\label{eq:secondterm}
    \P_{x,\pi^x_{p}}(d(X_{\ttau_k},y)\leq k^{\frac{1}{4d}})\leq C_1\frac{ k^{1/4}}{k^{d/2}}.
\end{equation}
Since the estimates \eqref{eq:firstterm} and \eqref{eq:secondterm} are summable for $d\geq3$, there exists $K=K(d,p)$ and~$C_2=C_2(d,p)$ such that
\begin{equation}
    \P_{x,\pi^x_{p}}\left(\bigcup_{k=K}^{\lceil\frac{1}{2}n^2\rceil-1} E_k\right)\leq\sum_{k=K}^{\lceil\frac{1}{2}n^2\rceil-1}\P_{x,\pi^x_{p}}(E_k)<C_2<1.
\end{equation}
By choosing $K$ to be large, we can make $C_2$ arbitrarily small, in particular smaller than $\frac{1}{6C_1}$. So with probability arbitrarily close to 1, $X$ does not visit $y$ during $[\ttau_K,\ttau_{\lceil\frac{1}{2} n^2\rceil}]$.

Now assume $d(x,y)\geq M$ for some universal constant $M=M(d,p)>0$. Let $A$ be the event that there exists $k\in\{0,\ldots,K\}$ with $d(X_{\ttau_k},y)\leq\frac{M}{2}$. Then applying Doob's $L^2$ inequality to the martingale $(X_{\ttau_k}-X_0)_k$ we get 
\begin{equation}
    \begin{split}
        \P_{x,\pi_{p}^x}(A)\leq&\P_{x,\pi_{p}^x}\left(\max_{k\in\{0,\ldots,K\}}\|X_{\ttau_k}-X_0\|_2\geq\frac{M}{2}\right)\\
        \leq&\frac{4\E_{x,\pi_{p}^x}[\|X_{\ttau_K}-X_0\|_2^2]}{\frac{M^2}{4}}=\frac{16K\E_{x,\pi_{p}^x}[\|X_{\ttau_1}-X_0\|_2^2]}{M^2},
    \end{split}
\end{equation}
where for the last equality we used the orthogonality of the increments. Using now Lemma~\ref{lemma:exponentialtailsR}, we see that this last bound goes to $0$ uniformly over $n$ and $\mu$ as $M\rightarrow\infty$. 

Let $B$ be the event that there exists $k\in\{0,\ldots,K\}$ with $|\cR[\ttau_k,\ttau_{k+1}]|\geq\frac{M}{2}$. Then by a union bound and Lemma~\ref{lemma:exponentialtailsR} again we deduce
\begin{equation}
    \begin{split}
        \P_{x,\pi_{p}^x}(B)\leq&\sum_{k=0}^K\P_{x,\pi_{p}^x}\left(|\cR[\ttau_k,\ttau_{k+1}]|\geq\frac{M}{2}\right)\leq (K+1)C_{\cR,2}e^{-C_{\cR,1}\frac{M}{2}},
    \end{split}
\end{equation}
which also tends to 0 uniformly over $n$ and $\mu$ as $M\rightarrow\infty$. Combining these two estimates, we obtain that the probability that $X$ visits $y$ up until time $\ttau_K$ can be made arbitrarily small by choosing $M$ large enough.

Lastly, using Lemma \ref{lemma:localCLT}, we have for all $k\geq \frac{1}{2} n^2$ 
\begin{equation}
    \P_{x,\pi^x_{p}}(d(X_{\ttau_k},y)\geq\tfrac{1}{2}n)\geq \frac{1}{2C_1},
\end{equation}
which implies that
\begin{equation}
    \P_{x,\pi^x_{p}}\left(\widetilde{\sigma}\leq \lceil\frac{1}{2}n^2\rceil\frac{C_{A}}{\mu}\right)\geq \frac{1}{2C_1}.
\end{equation}

Now define 
\begin{equation}
    \begin{split}
        A_1=&\{X\text{ does not visit }y\text{ during }[0,\ttau_K]\},\quad A_2=\{X\text{ does not visit }y\text{ during }[\ttau_K,\ttau_{\lceil\frac{1}{2} n^2\rceil}]\},\\
        A_3=&\{\widetilde{\sigma}\leq\lceil\frac{1}{2}n^2\rceil\frac{C_{A}}{\mu}\}.
    \end{split}
\end{equation}
Then, combining all prior estimates, we obtain
\begin{equation}\label{eq:donothitybeforesigmatilde}
    \begin{split}
        \P_{x,\pi^x_{p}}(X\text{ does not visit }y\text{ during }[0,\widetilde{\sigma}])\geq&\P_{x,\pi^x_{p}}(A_1\cap A_2\cap A_3)\\
        \geq&\P_{x,\pi_{p}^x}(A_3)-P_{x,\pi_{p}^x}(A^c_1)-P_{x,\pi_{p}^x}(A^c_2)\\
        \geq&\frac{1}{2C_1}-\frac{1}{6C_1}-\frac{1}{6C_1}=\frac{1}{6C_1}.
    \end{split}
\end{equation}
Finally, we want to get rid of the assumption that $d(x,y)\geq M$ for some $M$. Let $\widetilde{\sigma}_{M}=\inf\{\ttau_k\geq0:d(X_{\ttau_k},y)\geq M\}$. Let $x\in\Z_n^d$ such that $0<d(x,y)<M$ and let $\Gamma=\{v_0,\ldots,v_m,\}\subset \Z_n^d$ be a path from $x=v_0$ to some vertex $z=v_m$ with $d(z,y)\geq M$. Without loss of generality, $|\Gamma|\leq M+2$. Let~$A_j$ be the event that $X_{\ttau_j}=v_j$ and that after time $\ttau_j$ the following occur:
\begin{enumerate}
    \item the edge $e_j$ connecting $v_j$ and $v_{j+1}$ opens up before any of the other edges incident to $v_j$,
    \item before any of the other edges incident to $v_j$ or $v_{j+1}$ open up, the edge $e_j$ closes again, at which time the walk $X$ is at vertex $v_{j+1}$ and
    \item all edges incident to $v_{j}$ refresh before any of the edges incident to $v_{j+1}$ open up.
\end{enumerate}
It is immediate to see that, for all $j$,  we have $\P_{x,\pi_p^x}(A_j)>c>0$ for some constant $c=c(d,p)>0$ that is independent of $\mu$ and $n$. Furthermore, on the event $A_j$, $\cR[\ttau_j,\ttau_{j+1}]=\{v_j,v_{j+1}\}$. Hence,
\begin{equation}
    \begin{split}
        \P_{x,\pi_{p}^x}(\widetilde{\sigma}_M<\sigma_y)\geq\P_{x,\pi_{p}^x}\left(\bigcap_{j=0}^{m-1}A_j\right)\geq c^{M+2}.
    \end{split}
\end{equation}
For the last inequality we use the fact that, conditioning on the event that $X_{\ttau_j}=v_j$, the events $A_j$ are independent. Thus, using the strong Markov property at the stopping time~$\widetilde{\sigma}_M$ together with~\eqref{eq:donothitybeforesigmatilde} we are able to obtain our final estimate
\begin{equation}
    \begin{split}
        \P_{x,\pi_{p}^x}(\widetilde{\sigma}<\sigma_y)\geq&\P_{x,\pi_{p}^x}(\widetilde{\sigma}<\sigma_y\mid \widetilde{\sigma}_M<\sigma_y)\,\P_{x,\pi_{p}^x}(\widetilde{\sigma}_M<\sigma_y)\\
        \geq&\sum_{\substack{z\in\Z_n^d\\d(z,y)\geq M}}\P_{z,\pi_{p}^z}(\widetilde{\sigma}<\sigma_y)\P_{x,\pi_{p}^x}(X_{\widetilde{\sigma}_M}=z|\widetilde{\sigma}_M<\sigma_y) c^{M+2}\geq\frac{1}{6C_1}c^{M+2}>0.
    \end{split}
\end{equation}
This now concludes the proof.
\end{proof}

\begin{lemma}\label{lemma:hityuntilmixing}
For all $d\geq3$ and $p\in(0,p_c(d))$, there exists $c=c(d,p)>0$ such that for all $n\in\N$, $\mu\in(0,1)$ and $x,y\in\Z_n^d$ such that $x\neq y$,
\begin{equation}
    \P_{x,\pi^x_{p}}(y\in\cR[0,2t_{\mix}])<c<1.
\end{equation}
\end{lemma}

In the proof below, we first show that, after reaching distance $\frac{1}{2}n$ from $y$, the probability that the walk $(X_{\ttau_k})_{k\in\N_0}$ comes within distance $(\log n)^2$ of $y$ before time $2t_{\mix}$ is very small. We then show that, between times $\ttau_k$ and $\ttau_{k+1}$, the probability that the random walk travels a distance of at least $(\log n)^2$ is very small. This shows that the probability that the walk hits $y$ during the time interval $[\tsigma,2t_{\mix}]$ is small. By the previous lemma, the probability that the walk hits $y$ during the time interval $[0,\tsigma]$ is bounded away from 1, which completes the proof.

\begin{proof}
Let $A$ be the event that that there exists $k\in\N$ such that $\widetilde{\sigma}\leq\ttau_k\leq2 t_{\mix}$ and $d(X_{\ttau_k},y)\leq(\log n)^2$. Note that $\ttau_k\geq\frac{kC_{A}}{\mu}$ and $2t_{\mix}\leq\frac{C_1n^2}{\mu}$ for some constant $C_1=C_1(d,p)>0$ by Theorem~\ref{theorem:mixingupperbound}. Let $C_2=C_2(d,p)>0$ be some constant whose value we will specify later on in the proof. Then, using the strong Markov property at the stopping time $\tsigma$ in the second step below, we obtain
\begin{equation}
    \begin{split}
        \P_{x,\pi^x_{p}}(A)\leq&\max_{\substack{z\in\Z_n^d:d(z,y)\geq\frac{1}{2}n}}\P_{z,\pi_{p}^z}(\exists\ \ttau_{k}\in[0,2t_{\mix}]:d(X_{\ttau_k},y)\leq(\log n)^2)\\
        \leq&\max_{\substack{z\in\Z_n^d:d(z,y)\geq\frac{1}{2}n}}\P_{z,\pi^z_{p}}\left(\exists\ k\in\{0,\ldots,\lceil\frac{C_1}{C_{A}}n^2\rceil\}\colon d(X_{\ttau_k},y)\leq (\log n)^2\right)\\
        \leq&\max_{\substack{z\in\Z_n^d:d(z,y)\geq\frac{1}{2}n}}\P_{z,\pi^z_{p}}\left(\max_{k\in\{0,\ldots,\lceil C_2 n^2\rceil\}}\|X_{_{\ttau_k}}-X_0\|_2>\frac{n}{4}\right)\\
        &+\sum_{k=\lceil C_2n^2\rceil+1}^{\lceil\frac{C_1}{C_{A}}n^2\rceil}\max_{\substack{z\in\Z_n^d:d(z,y)\geq\frac{1}{2}n}}\P_{z,\pi^z_{p}}(d(X_{\ttau_k},y)\leq (\log n)^2).
    \end{split}
\end{equation}
We now apply Doob's maximal inequality to the martingale~$(X_{\ttau_k})_{k\geq0}$ for the first part and Lemma~\ref{lemma:localCLT} for the second part to obtain
\begin{equation}
    \begin{split}
        \P_{x,\pi^x_{p}}(A)\leq&\max_{\substack{z\in\Z_n^d:d(z,y)\geq\frac{1}{2}n}}\frac{4\E_{z,\pi^z_{p}}[\|X_{\ttau_{\lceil C_2 n^2\rceil}}-X_0\|^2_2]}{\frac{n^2}{16}}+\sum_{k=\lceil C_2n^2\rceil+1}^{\lceil\frac{C_1}{C_{A}}n^2\rceil}C_3(\log n)^{2d}\left(\frac{1}{k^{d/2}}\vee\frac{1}{n^d}\right)\\
        \leq&\frac{64\lceil C_2n^2\rceil\E_{0,\pi_{p}^0}[\|X_{\ttau_{1}}-X_0\|^2_2]}{n^2}+C_4\frac{(\log n)^2}{n^{d-2}}
    \end{split}
\end{equation}
for some constants $C_3=C_3(d,p)>0$ and $C_4=C_4(d,p,C_2)>0$. By first choosing $C_2$ small enough, and then $n$ large enough, the final bound may be made arbitrarily small, because $d\geq3$. The last inequality follows from the independence of the increments of $(X_{\ttau_k})_{k\geq1}$, and we may remove the maximum over $z$ by symmetry.

Furthermore, let $B$ be the event that there exists $k\in\N$ such that $\widetilde{\sigma}\leq\ttau_k\leq2 t_{\mix}$ and $|\cR[\ttau_k,\ttau_{k+1}]|\geq (\log n)^2$. Then, analogously to the previous estimate,
\begin{equation}
    \begin{split}
        \P_{x,\pi^x_{p}}(B)\leq&\sum_{k=0}^{\lceil\frac{C_3}{C_{A}}n^2\rceil}\max_{\substack{z\in\Z_n^d:d(z,y)\geq\frac{1}{2}n}}\P_{z,\pi^z_{p}}(|\cR[\ttau_k,\ttau_{k+1}]|\geq (\log n)^2)\\
        \leq&\left(\frac{C_3}{C_{A}}+2\right)n^2C_{\cR,2}e^{-C_{\cR,1}(\log n)^2}.
    \end{split}
\end{equation}

There exists $N=N(d,p)$ such that for all $n\geq N$, both expressions are smaller than\\ $\frac{1}{4}\P_{x,\pi^x_{p}}(\tilde{\sigma}<\sigma_y)>\frac{1}{4}c$, with $c=c(d,p)$ as in Lemma \ref{lemma:hittingboundary}. So,
\begin{equation}
    \begin{split}
        \P_{x,\pi^x_{p}}(y\not\in\cR[0,2t_{\mix}])\geq&\P_{x,\pi^x_{p}}(\{\tilde{\sigma}<\sigma_y\}\cap A^c\cap B^c)\\
        \geq&\P_{x,\pi^x_{p}}(\tilde{\sigma}<\sigma_y)-\P_{x,\pi^x_{p}}(A)-\P_{x,\pi^x_{p}}(B)\geq \frac{1}{2}c>0,
    \end{split}
\end{equation}
which completes the proof.
\end{proof}

\begin{lemma}\label{lemma:hityaftermixing}
For all $d\geq3$, $p\in(0,p_c(d))$ and $\delta\in(0,1)$, there exists a constant $C=C(d,p,\delta)>0$ such that for all $n\in\N$, $\mu\leq1$, $x\in\Z_n^d$ and $y\neq x$
\begin{equation}\label{eq:hityaftermixing}
    \P_{x,\pi^x_{p}}(y\not\in\cR[2t_{\mix},C\frac{n^d}{\mu}])>1-\delta.
\end{equation}
\end{lemma}

\begin{proof}
Using that $(M_t)_{t\geq0}$ is a time-homogeneous, reversible Markov process combined with \cite[Lemma 6.17]{levin2017markov}, we obtain
\begin{equation}
    \begin{split}
        &\P_{x,\pi^x_{p}}(y\in\cR[2t_{\mix},C\tfrac{n^d}{\mu}])\leq4\P_{u,\pi_p}(y\in\cR[0,C\tfrac{n^d}{\mu}-2t_{\mix}])=\frac{4}{n^d}\E_{u,\pi_p}[|\cR[0,C\tfrac{n^d}{\mu}-2t_{\mix}]|].
    \end{split}
\end{equation}

To bound this final expectation we now use an argument from the proof of \cite[Theorem 1.12]{Peres2015Random}. Since $p$ is strictly smaller than $p_c(d)$, there exists some $\beta=\beta(d,p)>0$ such that, when $\eta_0\sim\pi_p$, the probability that a fixed edge $e$ is open at some point during $[0,\frac{\beta}{\mu}]$ can be made strictly smaller than~$p_c(d)$. Let this probability be $p'=p'(\beta)$. Then the size of the set of vertices that the random walk can visit during $[0,\frac{\beta}{\mu}]$ can be stochastically dominated by the size of the open cluster of the origin of $p'$-percolation. Since $p'<p_c(d)$ and by stationarity, there exists a positive constant $C_1=C_1(d,p')<\infty$ such that for all $t\geq0$
\begin{equation}
    \E_{u,\pi_p}[|\cR[t,t+\tfrac{\beta}{\mu}]|]=\E_{u,\pi_p}[|\cR[0,\tfrac{\beta}{\mu}]|]\leq C_1.
\end{equation}
Now choosing $C$ small enough we get
\begin{equation}
    \begin{split}
        \frac{4}{n^d}\E_{u,\pi_p}[|\cR[0,C\tfrac{n^d}{\mu}-2t_{\mix}]|]\leq&\frac{4}{n^d}\sum_{j=0}^{\lceil (Cn^d-2t_{\mix})/\beta\rceil-1}\E_{u,\pi_p}[|\cR[2t_{\mix}+j\beta/\mu,2t_{\mix}+(j+1)\beta/\mu]|]\\
        \leq&\frac{4}{n^d}\frac{Cn^d}{\beta}C_{1}<\delta,
    \end{split}
\end{equation}
which completes the proof.
\end{proof}

\begin{proof}[Proof of Theorem \ref{th:mainhitting}]
Combining Lemmas \ref{lemma:hityuntilmixing} and \ref{lemma:hityaftermixing}, we obtain that there exist constants $C=C(d,p)>0$ and $c=c(d,p)>0$ such that
\begin{equation}
    \begin{split}
        \P_{x,\pi^x_{p}}\left(\sigma_y\geq C\frac{n^d}{\mu}\right)\geq\P_{x,\pi^x_{p}}\left(y\not\in\cR[2t_{\mix},C\frac{n^d}{\mu}]\right)-\P_{x,\pi^x_{p}}\left(y\in\cR[0,2t_{\mix}]\right)\geq c>0.
    \end{split}
\end{equation}
Hence,
\begin{equation}
    \E_{x,\pi_{p}^x}[\sigma_y]\geq C\frac{n^d}{\mu}\P_{x,\pi^x_{p}}\left(\sigma_y\geq C\frac{n^d}{\mu}\right)\geq c\cdot C\frac{n^d}{\mu},
\end{equation}
which finishes the proof.
\end{proof}

\subsection{Matthews' method}\label{sec:matthews}
We are now ready to finish the proof of the cover time lower bounds for $d\geq3$.

\begin{proof}[Proof of lower bound of Theorem~\ref{th:maincover} for $d\geq3$]
To prove the theorem it suffices to show that if $A$ is the set of vertices in $\Z_n^d$ whose coordinates are multiples of $\lfloor\sqrt{n}\rfloor$, then there exists a constant  $C=C(d,p)>0$ such that for all $n\in\N$ and $\mu\leq1$,
\begin{equation}\label{eq:goalwitha}
    C\frac{n^d}{\mu}\left(1+\ldots+\frac{1}{|A|-1}\right)\leq t_{\cov}.
\end{equation}
Label the vertices of $A$ as $\{1,\ldots,|A|\}$ and let $\phi$ be a uniformly random permutation of $\{1,\ldots,|A|\}$ that is independent of $((X_t,\eta_t))_{t\geq0}$. We define $T_k$ and $L_k$ as in the proof of the upper bound. Then $\P_{x,\pi_{p}^x}(L_k=\phi(k))=\frac{1}{k}$ by independence of $\phi$ from $X$. Furthermore, the event $L_k\neq\phi(k)$ is equivalent to the event $\phi(k)\in\cR[0,T_{k-1}]$, which implies $T_k=T_{k-1}$. Let $k\geq2$ and define $J=\inf\{j\geq0\colon T_{k-1}<\ttau_j<T_k\}$. Note that $J$ may be infinite. Then
\begin{equation}\label{eq:matthewslowerfirststep}
    \begin{split}
        \E_{x,\pi^x_{p}}[T_k-T_{k-1}| L_k=\phi(k)]\geq\sum_{j=0}^{\infty}\E_{x,\pi^x_{p}}[T_k-T_{k-1}|L_k=\phi(k),\,J=j]\,\P_{x,\pi^x_{p}}(J=j| L_k=\phi(k))\\
        \end{split}
\end{equation}
Observe that $\{L_k=\phi(k)\},\{J=j\}\in\cF_{\ttau_j}$, where $(\cF_t)_{t\geq0}$ is the natural filtration of $(M_t)_{t\geq0}$ which also keeps track of all the refresh times. Then, using the strong Markov property and the fact that at time $\ttau_{j}$, conditioned on $X_{\ttau_j}=z$, the environment is distributed as $\pi^z_{p}$, we obtain that the first term can be bounded as
\begin{equation}
    \begin{split}
        \E_{x,\pi^x_{p}}[T_k-T_{k-1}\mid L_k=\phi(k),\,J=j]\geq&\min_{\substack{y,z\in\Z_n^d\\y\neq z}}\E_{z,\pi_{p}^z}[\sigma_y]\geq C_1\frac{n^d}{\mu}
    \end{split}
\end{equation}
for some constant $C_1=C_1(d,p)>0$ by Theorem \ref{th:mainhitting}. Equation \eqref{eq:matthewslowerfirststep} then becomes
\begin{equation}
    \begin{split}
        \E_{x,\pi^x_{p}}[T_k-T_{k-1}\mid L_k=\phi(k)]\geq C_1\frac{n^d}{\mu}\,\P_{x,\pi^x_{p}}(\exists j\colon T_{k-1}<\ttau_j<T_k\mid L_k=\phi(k)).
    \end{split}
\end{equation}
We show that the latter probability is bounded away from zero. Using that the points in $A$ are at least distance $\sqrt{n}$ apart and that $\ttau_0=0$ we have
\begin{equation}
    \begin{split}
        &\P_{x,\pi^x_{p}}(\exists j\colon T_{k-1}<\ttau_j<T_k\mid L_k=\phi(k))\\
        \geq&\P_{x,\pi^x_{p}}\left(\left\{\exists j\colon\ttau_j\leq 2t_{\cov}\text{ and }|\cR[\ttau_j,\ttau_{j+1}]|\geq\frac{1}{2}\sqrt{n}\right\}^c\cap\{T_{k-1}\leq2t_{\cov}\}\right).
        \end{split}
        \end{equation}
      By Markov's inequality we now deduce  
    \begin{equation}
            \begin{split}
                &\P_{x,\pi^x_{p}}(\exists j\colon T_{k-1}<\ttau_j<T_k\mid L_k=\phi(k))\geq \frac{1}{2}-\P_{x,\pi^x_{p}}(\exists j\colon\ttau_j\leq 2t_{\cov}\text{ and }|\cR[\ttau_j,\ttau_{j+1}]|\geq\frac{1}{2}\sqrt{n})\\
        \geq&\frac{1}{2}-\sum_{j=0}^{\lceil\frac{2\mu t_{\cov}}{C_{A}}\rceil}\P_{x,\pi^x_{p}}(|\cR[\ttau_j,\ttau_{j+1}]|\geq\frac{1}{2}\sqrt{n})\geq\frac{1}{2}-\frac{2C_2}{C_{A}}n^d(\log n)C_{\cR,2}e^{-\frac{1}{2}C_{\cR,1}\sqrt{n}}\geq c_3>0
    \end{split}
\end{equation}
for all $n$ large enough and some constants $C_2=C_2(d,p)>0$ and $c_3=c_3(d,p)>0$, where for the second inequality we used that $\ttau_{j+1}-\ttau_{j}\geq\frac{C_{A}}{\mu}$ and a union bound. For the third inequality, we use the upper bound for $t_{\cov}$ that was already proven, and Lemma \ref{lemma:exponentialtailsR}.

Combining the bounds above, we obtain that there exists $c_4=c_4(d,p)>0$ such that 
\begin{equation}
    \E_{x,\pi^x_{p}}[T_k-T_{k-1}\mid L_k=\phi(k)]\geq c_4\frac{n^d}{\mu}.
\end{equation}
Furthermore, we have $\phi(1)=x$ with probability $\frac{1}{|A|}$, in which case $T_1=0$. If $\phi(1)\neq x$, then the same bound as above also holds for $k=1$. The rest of the proof now follows as for the upper bound. We conclude that
\begin{equation}
    \begin{split}
        t_{\cov}\geq&\E_{x,\pi^x_{p}}[\sigma_{\cov}]\geq\E_{x,\pi^x_{p}}[T_{|A|}]=\E_{x,\pi^x_{p}}[T_1]+\sum_{k=2}^{|A|}\E_{x,\pi^x_{p}}[T_k-T_{k-1}]\\
        =&\E_{x,\pi^x_{p}}[T_1|\phi(1)=x]\P(\phi(1)=x)+\E_{x,\pi^x_{p}}[T_1|\sigma(1)\neq x]\P(\phi(1)\neq x)\\
        &+\sum_{k=2}^{|A|}\E_{x,\pi^x_{p}}[T_k-T_{k-1}|L_k=\phi(k)]\P(L_k=\sigma(k))\\
        \geq&C_4\frac{n^d}{\mu}\left(1-\frac{1}{|A|}\right)+\sum_{k=2}^A\frac{1}{k}C_4\frac{n^d}{\mu}=C_4\frac{n^d}{\mu}\left(1+\ldots+\frac{1}{|A|-1}\right).
    \end{split}
\end{equation}
This concludes the proof of~\eqref{eq:goalwitha} and thus the proof of the theorem.
\end{proof}

\section{Lower bound for $d=2$}\label{sec:d2}
In dimension 2, the walk $X$ is no longer transient. Instead, we first prove strong approximation of~$X$ by Brownian motion. We can then transfer known results for cover times of Brownian motion in two dimensions to cover times for the random walk. Note that weak convergence of $X$ to Brownian motion was already shown in \cite[Theorem 3.1]{Peres2015Random}, but that we need a stronger result for our purposes. The proof follows the strategy of \cite[Section 4]{Dembo2004cover}.

\begin{proof}[Proof of the lower bound in Theorem \ref{th:maincover}, $d=2$]
Consider random walk on dynamical percolation on $\Z^2$ with parameters $\mu>0$ and $p\in (0,1)$ started from $\delta_0\times\pi_{p}^x$. Recall the definition of the regeneration times from Section~\ref{sec:preliminaries}. Note that $\tau_k$, $\ttau_k$ and $C_A$ were originally defined only for the process on $\Z_n^d$, but all the definitions and results carry over to the setting of $\Z^d$ with no changes. Let~$\sigma^2$ be the variance of the first coordinate of $X_{\ttau_1}$. Recall that the process $(\sigma^{-1}X_{\ttau_k})_{k\in\N}$ is a symmetric Markov process with covariance matrix being the identity matrix.
Furthermore, by Lemma \ref{lemma:exponentialtailsR}, there exists $\alpha$ that is independent of $\mu$ such that 
\begin{equation}
    \alpha\cdot \E_{0,\pi^0_{p}}\left[|\sigma^{-1}X_{\ttau_1}|e^{\alpha|\sigma^{-1}X_{\ttau_1}|}\right]\leq1.
\end{equation}
Hence, by \cite[Theorem 12]{Einmahl1989extensions}, we may construct $(X_{\ttau_k})_{k\in\N}$ and a Brownian motion $W$ on the same probability space such that for all $n$ and $x\geq0$,
\begin{equation}\label{eq:strongapprox}
    \P\left(\max_{1\leq k\leq n}\left|\sigma^{-1}X_{\ttau_k}-W_k\right|\geq x\right)\leq c_1n\left[e^{-c_2\alpha x}+e^{-c_2\left(\frac{x}{c_3}\right)^{1/2}}\right],
\end{equation}
where $c_1$, $c_2$ and $c_3$ are independent of $\mu$.

Let $\delta>0$, $\gamma\in(0,1)$ and let $C_1$ be some positive constant that we will specify later. By scale invariance of Brownian motion we obtain
\begin{equation}\label{eq:brownianapproximation}
    \begin{split}
        &\P_{0,\pi_{p}^0}\left(\max_{k\leq C_1n^2(\log n)^2}\sup_{\ttau_{k-1}\leq t\leq \ttau_k}\left|W_{\frac{\sigma^2 k}{n^2}}-\frac{1}{n}X_t\right|\geq \sigma n^{\gamma-1}\right)\\
        \leq&\P_{0,\pi_{p}^0}\left(\max_{k\leq C_1n^2(\log n)^2}\left|W_{k}-\sigma^{-1}X_{\ttau_k}\right|\geq \frac{1}{2}n^{\gamma}\right)\\
        &+\P_{0,\pi_{p}^0}\left(\max_{k\leq C_1n^2(\log n)^2}\sup_{\ttau_{k-1}\leq t\leq\ttau_k}\left|\sigma^{-1}X_{\ttau_k}-\sigma^{-1}X_t\right|\geq \frac{1}{2}n^{\gamma}\right)\\
        \leq&c_1C_1n^2(\log n)^2\left[e^{-c_2\alpha n^{\gamma}}+e^{-c_2\left(\frac{n^\gamma}{c_3}\right)^{1/2}}\right]+C_1n^2(\log n)^2C_{\cR,2}e^{-C_{\cR,1}\sigma n^\gamma}<\delta
    \end{split}
\end{equation}
for all $n$ large enough, where for the last bound we used~\eqref{eq:strongapprox} and Lemma~\ref{lemma:exponentialtailsR} as in~\eqref{eq:firstterm}.

Let $\mathcal{C}_{\varepsilon}$ be the cover time of the Wiener sausage of radius $\varepsilon$ on the unit torus in 2 dimensions. It was shown in \cite{Dembo2004cover} that
\begin{equation}
    \lim_{\varepsilon\downarrow0}\frac{\mathcal{C}_\varepsilon}{(\log\varepsilon)^2}=\frac{2}{\pi}
\end{equation}
in probability. Hence, for every $\eta\in(0,1)$,
\begin{equation}
    \P\left(\mathcal{C}_{\varepsilon}>\frac{2}{\pi}(1-\eta)(\log\varepsilon)^2\right)\geq1-\delta
\end{equation}
for $\varepsilon$ small enough. So with probability $1-\delta$, some disc of radius $\varepsilon$ is missed by the set
\begin{equation}
    \left\{W_t\bmod\Z^2:t\leq\frac{2}{\pi}(1-\eta)(\log\varepsilon)^2\right\}
\end{equation}
for all $\varepsilon$ small enough. Choosing $\varepsilon=\varepsilon_n=2\sigma n^{\gamma-1}$ and $\eta=\frac{1}{2}$, we obtain that with probability $1-\delta$, the set
\begin{equation}
    \left\{W_{\frac{\sigma^2k}{n^2}}\bmod\Z^2:k\leq\frac{1}{\pi}(\log(2\sigma n^{1-\gamma}))^2\sigma^{-2}n^2\right\}
\end{equation}
misses a disc of radius $\varepsilon_n$ for all $n$ large enough. By \eqref{eq:brownianapproximation}, choosing $C_1$ small enough that $C_1n^2(\log n)^2<\frac{1}{\pi}\log(2\sigma n^{1-\gamma}))^2\sigma^{-2}n^2$, we conclude that with probability $1-2\delta$, the set
\begin{equation}
    \left\{\frac{1}{n}X_t\bmod\Z^2:t\leq\ttau_{C_1n^2(\log n)^2}\right\}
\end{equation}
misses a disc of radius $\frac{1}{2}\varepsilon_n=\sigma n^{\gamma-1}$ for all $n$ large enough. Since $\ttau_k\geq k\cdot C_A/\mu$ for all $k$, we get
\begin{equation}
     \left\{\frac{1}{n}X_t\bmod\Z^2:t\leq\ttau_{C_1n^2(\log n)^2}\right\} \subseteq   \left\{\frac{1}{n}X_t\bmod\Z^2:t\leq\frac{C_AC_1n^2(\log n)^2}{\mu}\right\}
\end{equation}
and this now completes the proof.
\end{proof}

\begin{remark}
We note that the above method of proof would not yield sharp bounds in higher dimensions. Indeed, for dimensions $d\geq3$, it was proved in \cite{Dembo2003brownian} that
\begin{equation}
    \lim_{\varepsilon\downarrow0}\frac{\mathcal{C}_\varepsilon}{\varepsilon^{2-d}\log\frac{1}{\varepsilon}}=\kappa_d,
\end{equation}
with $\kappa_d$ some constant depending only on $d$. Using the same strategy as in the proof above, this would only give
\begin{equation}
    t_{\cov}\gtrsim_\gamma\frac{n^{d+\gamma(2-d)}\log n}{\mu}.
\end{equation}
for all $\gamma\in(0,1)$. 
\end{remark}

\section{The case $d=1$}\label{sec:d1}

\begin{proof}[Proof of Theorem \ref{th:maincover} for $d=1$]
The lower bound follows immediately from the lower bound on the maximum expected hitting time \cite[Theorem 1.12]{Peres2015Random}.

We now turn to the upper bound. Without loss of generality, assume that $n$ is even. Assume for the moment that the walk starts from $\delta_0\times\pi_{p}^0$. Define $\zeta_i$ and $\xi_i$ as follows. We let $\zeta_0=0$ and for $i\geq1$,
\begin{equation}
    \begin{split}
        \zeta_i:=&\min\{\ttau_k>\xi_{i-1}:\,X_{\ttau_k}=\frac{n}{2}\},\\
        \xi_i:=&\min\{\ttau_k>\zeta_{i}:\,X_{\ttau_k}=0\}.
    \end{split}
\end{equation}
Note that in each of the time intervals $[\xi_{i-1},\zeta_i]$ and $[\zeta_i,\xi_i]$, $X$ must have visited all of the vertices in one of the sets $\{0,\ldots,\frac{n}{2}\}$ and $\{\frac{n}{2},\ldots,n-1\}$. Let $A_i$ be the event that $X$ has visited all of the vertices in $\{0,\ldots,\frac{n}{2}\}$ during the interval $[\xi_{i-1},\zeta_i]$, and let $B_i$ be the event that $X$ has visited all of the vertices in $\{\frac{n}{2},\ldots,n-1\}$ during the interval $[\zeta_i,\xi_i]$. By the strong Markov property for the stopping times $\zeta_i$ and $\xi_i$, all of the events $A_i$ and $B_i$ are independent. Furthermore, by symmetry, $\P_{0,\pi_{p}^0}(A_i)=\P_{0,\pi_{p}^0}(B_i)\geq\frac{1}{2}$, so $\P_{0,\pi_{p}^0}(A_i\cap B_i)\geq\frac{1}{4}$.

Define
\begin{equation}
    N:=\min\{i\geq1:A_i\cap B_i\text{ occurs}\}.
\end{equation}
Then by the previous remarks, starting from $\delta_0\times\pi_{p}^0$, $N$ is stochastically dominated by a geometric random variable with success probability $\frac{1}{4}$. Also, by \cite[Lemma 7.2]{Peres2015Random} together with Lemma \ref{lemma:taubounded} and \eqref{eq:deftautilde}, there exists a constant $C_1=C_1(d,p)>0$ such that for all $i\geq1$,
\begin{equation}
    \E_{0,\pi_{p}^0}[\zeta_i-\xi_{i-1}]=\E_{0,\pi_{p}^0}[\xi_i-\zeta_{i}]\leq C_1\frac{n^2}{\mu}.
\end{equation}

By symmetry and the strong Markov property at $\ttau_1$,
\begin{equation}\label{eq:coverd1upperbound}
    \begin{split}
        \max_{x,\eta_0}\E_{x,\eta_0}[\tau_{\cov}]\leq\max_{x,\eta_0}\E_{x,\eta_0}[\ttau_1]+\E_{0,\pi_{p}^0}[\tau_{\cov}].
    \end{split}
\end{equation}
In \cite[Proposition 6.14]{Peres2015Random}, it is shown that there exists a constant $C_2=C_2(d,p)>0$ such that 
\begin{equation}\label{eq:tau1upperbound}
    \E_{0,\eta_0}[\ttau_1]\leq C_2\frac{\log n}{\mu}.
\end{equation}
Furthermore, by Wald's identity,
\begin{equation}\label{eq:coverd1Wald}
    \begin{split}
        \E_{0,\pi_{p}^0}[\sigma_{\cov}]\leq&\E_{0,\pi_{p}^0}\left[\sum_{i=1}^N(\xi_i-\xi_{i-1})\right]\leq\E_{0,\pi_{p}^0}[N]2C_1\frac{n^2}{\mu}\leq 4C_1\frac{n^2}{\mu}.
    \end{split}
\end{equation}
Combining \eqref{eq:coverd1upperbound}--\eqref{eq:coverd1Wald}, we conclude the proof.
\end{proof}

\bibliographystyle{plain}
\bibliography{referenties}

\paragraph{Acknowledgements.} This work was initiated as a Part III essay at the University of Cambridge under the supervision of Perla Sousi. The author thanks her for her guidance, the many mathematical discussions and her helpful comments on this paper. This work was supported by the University of Cambridge Harding Distinguished Postgraduate Scholarship Programme.

\end{document}